\documentclass[11pt]{article}
\usepackage{latexsym}
\usepackage{amsfonts}
\usepackage{amsmath}
\usepackage{amsmath,amssymb,proof1,latexsym}
\usepackage[all]{xypic}

\makeindex

\newcommand{\imp}{\!\rightarrow\!}

\newcommand{\mpn}{\medskip\par\noindent}

\newcommand{\pa}{{\sf PA}}

\newcommand{\proves}{\vdash}

\newcommand{\gn}[1]{\ulcorner {#1} \urcorner }

\newcommand{\lc}[1]{#1\!\!:\!\!}

\newtheorem{Prop}{\bf Proposition}

\newtheorem{Theor}{\bf Theorem}

\newtheorem{Lemma}{\bf Lemma}

\newtheorem{Coro}{\bf Corollary}

\newtheorem{Fact}{\bf Fact.}

\newtheorem{Remark}{\bf Remark}

\newtheorem{Claim}[enumi]{Claim}

\newtheorem{defin}{\bf Definition}

\newtheorem{exam}{\bf Example}

\newtheorem{notat}{\bf Notation.}

\newenvironment{proof}{{\bf Proof.}}{\hfill $\slot$}
\newcommand{\slot}{\hfill \mbox{$\Box$}\vspace{\parskip}\\}
\newtheorem{Comment}{\bf Comment}

\begin{document}
\title{Consistency formula is strictly stronger in \pa\ than \pa-consistency}

\author{Sergei Artemov\\ \\
 {\small The Graduate Center, the City University of New York}\\
{\small  365 Fifth Avenue, New York City, NY 10016}\\
{\small {\tt sartemov@gc.cuny.edu}} }
\date{\today}
\maketitle

\begin{abstract}
In this note, we show that, despite the wide-spread assumption, the consistency formula for Peano Arithmetic \pa, ${\sf Con}_{\sf PA}$, {\it $\forall x$ ``$x$ is not a code of a derivation of $(0=1)$,"} is not equivalent in \pa\ to the consistency of \pa. Specifically, we demonstrate that ``\pa\ is consistent" is provably in \pa\ equivalent to the series ${\sf Con}^S_{\sf PA}$ of arithmetical sentences {\it ``$n$ is not a code of a derivation of $\ (0=1)$"} for $n=0,1,2,\ldots$.  Since ${\sf Con}_{\sf PA}$ is strictly stronger in \pa\ than ${\sf Con}^S_{\sf PA}$, the unprovability of ${\sf Con}_{\sf PA}$ in \pa\ does not yield the unprovability of \pa-consistency. 

In \cite{Art25},  the \pa-consistency in the form ${\sf Con}^S_{\sf PA}$ has been proven in \pa. 

\end{abstract}

\section{Arithmetic: contentual vs. formal}

We consider some fragment of contentual mathematics which, within this paper, we call {\sf Ar} (for ``Arithmetic") which contains the set of natural numbers $\omega$
$$0,1,2,\ldots,n,\ldots $$
and includes all primitive recursive operations on $\omega$.  


We also consider the Peano Arithmetic \pa, which is a formal first-order theory with the symbols and defining identities for all primitive recursive functions, one sort of variables to symbolically represent unspecified ``natural numbers." Along with logical axioms, \pa\ has the Induction Principle represented by an infinite series of instances of induction for each arithmetical formula. \pa\ has a unique standard model which is the structure with the domain $\omega$ in the language of \pa\ satisfying all postulates of \pa. 

By G\"odel's Incompleteness Theorem, there are models of \pa\ which are not isomorphic to the standard model. Simply speaking, there are (continuum many) countable nonstandard models of \pa\ each of which, in addition to $\omega$, has so-called nonstandard numbers.
\begin{quote}  The domain of formal quantifiers $\forall x$ and $\exists x$ in \pa\ is not specified and depends on a model. In nonstandard models, $\forall x$ and  $\exists x$ {\bf do not range over  $0,1,2,\ldots,n,\ldots$}.
\end{quote}

\section{Contentual property vs. its formalization}\label{cont}

Consider a contentual property $P$ of natural numbers expressible in {\sf Ar}. Suppose that $P$ can be formalized as a formula $P_f$ in the language of \pa\ by the following step-by-step formalization process: recursive operations on natural numbers in {\sf Ar} are directly represented by \pa-terms, propositional logic operations are straightforward, informal quantifiers ``for all natural numbers $n$" and ``exists a natural number $n$" are formalized as  $\forall x$ and $\exists x$. We refer to $P_f$ as \[ \mbox{\it the direct formalization of $P$.} \] 


\subsection{\pa\ consistency vs. the consistency formula ${\sf Con}_{\sf PA}$}\label{notequivalent}
 
The consistency of a formal theory means that all its formal derivations are free of contradictions. For the case of \pa, 
\begin{equation}\label{con}
 \mbox{\it for any derivation $D$, $D$ is not a proof of  $(0\!=\!1)$.}
\end{equation}
Derivations are finite syntactic objects, and their G\"odel codes are all standard natural numbers. 
Such numbers can be identified with numerals, terms $\overline{0},\overline{1},\overline{2},\ldots,\overline{n}\ldots$ in the language of {\sf PA}. 
We do not distinguish a syntactic object $X$ and its G\"odel number $\gn{X}$, a natural number $n$, and the corresponding numeral $\overline{n}$ when safe. 
Let also $\bot$ denote the arithmetic formula $(0\!=\! 1)$. 

Let $x\!:\!y$ denote the standard primitive recursive proof predicate in \pa, 
\begin{equation}\label{pp}
\mbox{\it ``x is code of a proof of a formula having code y."} 
\end{equation}
The proof predicate is represented by a quantifier-free \pa-formula which we also call $x\!:\!y$. 

Since the proof predicate is a direct copy of the definition of a formal proof, for each derivation $D$ and formula $\varphi$, the equivalence of ``$D$ is a proof of $\varphi$" and $\gn{D}\!:\!\gn{\varphi}$ is a straightforward combinatorial observation about G\"odel coding and it is uniformly provable in \pa. 

Following G\"odel, we reformulate \pa\ consistency (\ref{con}) in an equivalent form in the contentual arithmetic {\sf Ar}:
\begin{equation}\label{godelcon}
 \mbox{\it for any natural number $n$, $\neg\lc{n}\bot$.}
\end{equation}
This G\"odelized definition (\ref{godelcon}) of \pa-consistency is not yet in the language of \pa\ since it contains an informal quantifier ``{\it for any natural number $n$}" not expressible in \pa.  

The direct formalization of (\ref{godelcon}) is the consistency formula ${\sf Con}_{\sf PA}$, i.e., 
\begin{equation}\label{conform}
\forall x (\neg\lc{x}\bot).
\end{equation}

It would be a mistake to assume {\it a priori} that the consistency formula ${\sf Con}_{\sf PA}$ (\ref{conform}) adequately represents 
\pa-consistency property (\ref{godelcon}) in \pa\ because of different quantifications in (\ref{godelcon}) and (\ref{conform}).

The direct mathematical comparison of (\ref{godelcon}) and (\ref{conform})  is not possible since (\ref{godelcon}) is a well-defined mathematical property whereas (\ref{conform}) is not; the formal quantifier ``$\forall x$" in ${\sf Con}_{\sf PA}$ is not specified and its meaning depends on a choice of an arithmetical model. 
Comparing (\ref{godelcon}) and (\ref{conform}) formally within \pa\ is also not possible since (\ref{godelcon}) is not in the language of \pa. 
So, the choice of ${\sf Con}_{\sf PA}$ as representing {\it `` \pa\ is consistent"} in \pa\ has not had a mathematical justification.

As it follows from the results of the next Section~\ref{equivalent}, such a choice cannot be justified since, in \pa,  consistency formula ${\sf Con}_{\sf PA}$ is strictly stronger than \pa-consistency.

\subsection{\pa-consistency is equivalent in \pa\ to the consistency scheme}\label{equivalent}

As we have already mentioned, the G\"odelized definition (\ref{godelcon}) of \pa-consistency contains an informal quantifier ``{\it for any natural number $n$}" not expressible in \pa. However, we can emulate quantifiers over standard numbers by sets of formulas in \pa. 
\mpn
{\bf Consistency scheme} is the set ${\sf Con}^S_{\sf PA}$ of \pa-formulas

\begin{equation}\label{schemecon} \{ \neg\lc{0}\bot, \neg\lc{1}\bot,\neg\lc{2}\bot,\neg\lc{3}\bot\ldots  \}.
\end{equation}
Sets (sequences, schemes, etc.) of formulas are common objects in metamathematics of first-order theories. In particular, the induction principle is expressed in \pa\ by the Induction Scheme, which is a primitive recursive set of \pa-formulas.
A set $X$ of \pa--formulas holds under a given interpretation if each formula in $X$ holds.\footnote{The rigorous notion of a proof of a serial property in a theory was offered in \cite{Art25} (selector proofs) and studied in \cite{Art25,EG25}.} So, the rigorous reading of {\it ``${\sf Con}^S_{\sf PA}$ holds"} is 
\begin{equation}\label{holds}
\mbox{\it each formula in ${\sf Con}^S_{\sf PA}$ holds}. 
\end{equation}

All formulas in ${\sf Con}^S_{\sf PA}$ are closed quantifier-free sentences and hence ${\sf Con}^S_{\sf PA}$ is a well-defined mathematical property. This allows us to compare (\ref{godelcon}) and (\ref{schemecon}) mathematically.  Their comparison below amounts to a trivial observation, but we carry it on anyway to make sure it is formalizable in \pa. 

\begin{Prop}\label{one}
{\it \pa-consistency and ${\sf Con}^S_{\sf PA}$ are mathematically equivalent.}
\end{Prop}
\begin{proof} 

$(\Rightarrow)$. Suppose {\it ``\pa\ is consistent,"} i.e., (\ref{godelcon}). Pick an arbitrary natural number $n$, then, by (\ref{godelcon}), $\neg\lc{n}\bot$ holds, hence 
${\sf Con}^S_{\sf PA}$ holds. 

$(\Leftarrow)$. Suppose ${\sf Con}^S_{\sf PA}$ holds, then each $\neg\lc{n}\bot$ holds, hence (\ref{godelcon}).

\end{proof}
The next step is to formalize this proof in \pa.
\begin{Prop}\label{two} 
{\it  The proof of Proposition~\ref{one} is formalizable in \pa.}
\end{Prop}
\begin{proof} This is a G\"odelian formalization under which syntactic objects, such as numerals, terms, formulas, are represented by their G\"odel numbers, operations and relations on these objects by primitive recursive terms and predicates. Furthermore, after G\"odel, we use the \pa-formula $\lc{x}y$ as a direct formal representation of the proof predicate (\ref{pp}). 

Given an arithmetic formula $\varphi(x)$, consider the natural arithmetical term  $\varphi(x)^\bullet$ for the primitive recursive function 
which for each $n$ returns $\gn{\varphi(n)}$.  In particular, $(\neg\lc{x}\bot)^\bullet$ is a primitive recursive enumeration of ${\sf Con}^S_{\sf PA}$ which for each $n$ return the G\"odel number of the $n$th formula in ${\sf Con}^S_{\sf PA}$. 

Consider a function which is, in a sense, the converse of $(\neg\lc{x}\bot)^\bullet$. Let $N(y)$ be a natural arithmetical term for a primitive recursive function 
\[ N(k)=\left\{ \begin{array}{ll}
n,  &  \mbox{\it if $k=\gn{\neg\lc{n}\bot}$}\\
0, & \mbox{\it otherwise}.
\end{array}
\right.
\]
By the construction of $N(y)$, 
\begin{equation}\label{N}  
\pa\proves\ \forall x\ ( N[(\neg\lc{x}\bot)^\bullet]=x).
\end{equation}
Consider a formula $S(y)$ which is $y=g(N(y))$ with $g(x)=(\neg\lc{x}\bot)^\bullet$. \pa\ can prove that $S(y)$ is a characteristic function for the set of G\"odel numbers of formulas from ${\sf Con}^S_{\sf PA}$, i.e.,
\begin{equation}\label{char}  
\pa\proves\ S(y) \ \leftrightarrow\ \exists x (y=(\neg\lc{x}\bot)^\bullet).
\end{equation}
Indeed, argue in \pa. Suppose $S(y)$, i.e., $y=g(N(y))$, and take $a=N(y)$. Then $y=(\neg\lc{a}\bot)^\bullet$ hence $\exists x (y=(\neg\lc{x}\bot)^\bullet)$. Now suppose for some $a$, $y=(\neg\lc{a}\bot)^\bullet$. By (\ref{N}), $N(y)=a$, hence $y=g(N(y))$, i.e., $S(y)$. 

Now we are back to proving Proposition~\ref{two}. 
To convert G\"odel numbers  $\gn{\neg\lc{n}\bot}$ back to formulas $\neg\lc{n}\bot$ in \pa, we use the truth predicate $\mbox{\it Tr}_1$ for $\Sigma_1$-formulas (cf. \cite{Art25,vO99}). This truth predicate satisfies Tarksi's condition for a $\Sigma_1$-formula $\varphi$:
\begin{equation}\label{tarski} 
\pa\proves \varphi(x)\leftrightarrow \mbox{\it Tr}_1(\varphi(x)^\bullet). 
\end{equation}
The direct formalization of \pa-consistency (\ref{godelcon}) is \[ \forall x(\neg\lc{x}\bot).\] The direct formalization of \emph{``${\sf Con}^S_{\sf PA}$ holds"} (\ref{holds}) is 
\[ \forall y(S(y)\imp\mbox{\it Tr}_1(y)). \]
Now we are formalizing the proof of Proposition~\ref{one} in \pa. 

$(\Rightarrow)$. Suppose $\forall x(\neg\lc{x}\bot)$. Pick $y$ and suppose $S(y)$. By (\ref{char}), there is $x$ such that $y=(\neg\lc{x}\bot)^\bullet$. Since 
$\neg\lc{x}\bot$ and (\ref{tarski}), $\mbox{\it Tr}_1((\neg\lc{x}\bot)^\bullet)$, hence $\mbox{\it Tr}_1(y)$. 

$(\Leftarrow)$. Suppose $\forall y(S(y)\imp\mbox{\it Tr}_1(y))$ and pick an arbitrary $x$. Take $y=(\neg\lc{x}\bot)^\bullet$. By (\ref{char}), $S(y)$. By assumptions, 
$\mbox{\it Tr}_1(y)$, hence $\mbox{\it Tr}_1((\neg\lc{x}\bot)^\bullet)$. By (\ref{tarski}), $\neg\lc{x}\bot$. 
\end{proof}
\mpn
{\bf Comment}. For the proof of Proposition~\ref{one} and its formalization in Proposition~\ref{two}, we have chosen natural and easy versions. However, the argument supports a more technical, but more constructive formulation.  In the current form, Proposition~\ref{two} proves that for appropriate $A$, $B$, 
\[ \pa\proves \forall x A(x) \leftrightarrow \forall x B(x) \] 
The proof of Proposition~\ref{two} actually shows something slightly stronger. It builds primitive recursive functions $f(x)$ and $h(x)$ such that 
\[ \pa\proves A(f(x))\imp B(x)\ \ \ \mbox{\it and}\ \ \   \pa\proves B(h(x))\imp A(x) .\]

\section{On Hilbert's consistency program}

Since ${\sf Con}^S_{\sf PA}$ is a mathematically justified way to represent the consistency of \pa\ in \pa, and ${\sf Con}_{\sf PA}$ is strictly stronger in \pa\ than ${\sf Con}^S_{\sf PA}$ (an easy observation, cf. \cite{Art25}), G\"odel's theorem stating the unprovability of ${\sf Con}_{\sf PA}$ does not answer the question of provability of \pa-consistency in \pa. The selector proof of the \pa-consistency in the form ${\sf Con}^S_{\sf PA}$ in \cite{Art25} provides a mathematically justified (affirmative) answer to the question of provability of \pa-consistency in \pa. 

For many decades, G\"odel's Second Incompleteness Theorem, yielding the unprovability of ${\sf Con}_{\sf PA}$ in \pa, has been regarded as mathematical evidence of the impossibility of Hilbert's consistency program (cf., for example, \cite{Kri22,Zach07}). The aforementioned analysis of differences between the consistency formula ${\sf Con}_{\sf PA}$ and \pa-consistency suggests that such cancelling of Hilbert's consistency program has been premature. 

For the recent progress in Hilbert's program related to treating consistency as a serial property, including the disproof of the unprovability of consistency paradigm,  cf. \cite{Art25,EG25}.

\end{document}